\documentclass[10pt]{article}
\usepackage{mathrsfs}
\usepackage{amsmath,amssymb}
\usepackage{amsthm}
\usepackage[colorlinks,linkcolor=red,anchorcolor=blue,citecolor=green]{hyperref}
\usepackage{graphics,graphicx}
\usepackage{color}
\usepackage{enumerate}
\usepackage{appendix}
\allowdisplaybreaks

\usepackage{tikz}

\usepackage{verbatim} 

\newtheorem{thm}{Theorem}[section]
\newtheorem{lem}[thm]{Lemma}
\newtheorem{prop}[thm]{Proposition}

\newtheorem{prob}[thm]{Problem}
\newtheorem{conj}[thm]{Conjecture}

\theoremstyle{definition}   

\def\qed{\hfill \rule{4pt}{7pt}}

\def\pf{\noindent {\it{Proof.}\hskip 2pt}}

\numberwithin{equation}{section}
\def\Z{{\mathbb{Z}}}
\def\G{{\mathbb{G}}}

\def\RW{{\mathrm{RW}}}

\definecolor{lw}{RGB}{255,0,0}
\definecolor{zs}{RGB}{0,0,255}
\newcommand{\lw}[1]{\textcolor{lw}{\bf #1}}
\newcommand{\zs}[1]{\textcolor{zs}{\bf #1}}
\definecolor{lightblue}{RGB}{0.68,0.85,0.9}
\definecolor{bluegray}{rgb}{0.4, 0.6, 0.8}
\newcommand{\rem}[1]{\textcolor{bluegray}{#1}}

\begin{document}
\begin{center}
{{\large\bf On Spectral Radius  of Biased Random Walks \\ on Infinite 
Graphs}\footnote{The project is supported partially by CNNSF (No.~11671216).}}
\end{center}
\vskip 2mm
\begin{center}
Z.
 Shi, V.
  Sidoravicius, H.
  Song, L. Wang,  K. Xiang

\end{center}


\vskip 2mm

\begin{abstract}


We consider a class of biased random walks on infinite graphs and present several general results on 
the spectral radius of
biased random walk.

\bigskip

\end{abstract}

\section{Introduction}
Let $G=(V(G), \, E(G))$ be a locally finite, connected infinite graph, where $V(G)$ is the set of its vertices and 
$E(G)$ is the set of its edges.  Fix a vertex $o$ of $G$ as {the} root.
For any reversible Markov chain on $G$, there is a stationary measure $\pi(\cdot)$ such that for any two adjacent vertices $x$ and $y$,  $\pi(x)p(x,\, y)=\pi(y)p(y,\, x)$, where $p(x,\, y)$ is the transition probability of the Markov
chain. For the edge joining vertices $x$ and $y$, we assign a weight
\[
c(x,y)
=
\pi(x)p(x, \, y) ,
\]

\noindent and call by {{\it conductances}} the weights of the edges. 
We study the biased random walks on the rooted graph $(G, \, o)$ defined as follows:

\noindent  For any vertex $x$ of $G$ let $|x|$ denote the graph distance between $x$ and $o$. Let 
$\mathbb{N} := \{ 1, \, 2, \ldots\}$ and  $\mathbb{Z}_{+}=\mathbb{N}\cup\{0\}$. For any $n\in\mathbb{Z}_{+}$:
$$
B_G(n)
=
\{x\in V(G):\ |x| \le n\},
\qquad
\partial B_G(n)
=
\{x\in V(G):\ |x| =n\}.
$$

\noindent Let $\lambda\in [0, \, \infty)$. If an edge $e=\{x,y\}$ is at distance $n$ from $o$, i.e., min$(|x|, |y|)=n$, its conductance is defined as $\lambda^{-n}$. Denote by $\RW_\lambda$  the nearest-neighbour random walk $(X_n)_{n=0}^\infty$ among such conductances and call it the $\lambda$-{\it biased random walk}. In other words, $\RW_{\lambda}$ 
has the following transition probabilities: for $v \sim u$ (i.e., if $u$ and $v$ are adjacent on $G$),
\begin{eqnarray}
    \label{(1.1)}
    p(v,u)
    :=
    p_{\lambda}^G(v,u)
  =
  \begin{cases}
    \frac{1}{d_v}
        &\mathrm{if}\ v=o,
        \\
        \frac{\lambda}{d_v+\left(\lambda-1\right)d_v^-}
        &\mathrm{if}\ u\in \partial B_G(|v|-1)\ \mathrm{and}\ v\neq o,
        \\
        \frac{1}{d_v+\left(\lambda-1\right)d_v^-}
        &\mathrm{otherwise}.
  \end{cases}
\end{eqnarray}

\noindent Here, $d_v$ is the degree of vertex $v$, and $d_v^-$, $d_v^0$ and $d_v^+$ are the numbers of edges connecting $v$ to $\partial B_G(|v|-1)$, $\partial B_G( |v| )$ and $\partial B_G(|v|+1)$ respectively. Note that
$$
d_v^{+}+d_v^0+d_v^-=d_v,
\qquad
d_v^{-}\ge 1,
\qquad
v\not=o,
\qquad
d_o^{-}=d_o^0=0,
$$

\noindent and that $\RW_{\lambda=1}$ is the simple random walk (SRW) on $G$.

By Rayleigh's monotonicity principle (see \cite{LR-PY2016}, p. 35), there is a critical value $\lambda_c(G)\in [0, \, \infty]$ such that $\RW_{\lambda}$ is transient for $\lambda<\lambda_c(G)$ and is recurrent for $\lambda>\lambda_c(G)$. Let $M_n=\#(\partial B_G(n))$ be the cardinality of $\partial B_G(n)$ for any $n\in\mathbb{Z}_{+}$. Define the volume growth rate of $G$ as
\[
\mathrm{gr}(G)
=
\liminf_{n\to \infty} M_n^{1/n}.
\]

\noindent When $G$ is a tree, $\lambda_c(G)$ is exactly the exponential of the Hausdorff dimension of the tree boundary, namely the branching number of the tree (\cite{Furst70}, \cite{LR1990}, \cite{LR-PY2016}). When $G$ is a transitive graph, $\lambda_c(G)=\mathrm{gr}(G)$ (see \cite{LR1995} and \cite{LR-PY2016}). Let
$$
\mathrm{gr}_+(G)
=
\liminf_{n\to\infty}\Big( \sum_{x\in\partial B_G(n-1)}d_x^{+}\Big)^{\! 1/n} .
$$

\noindent Clearly $\mathrm{gr}_+(G)\ge\mathrm{gr}(G)$.  If $G$ either is a tree or satisfies
$$\limsup_{n\to\infty}\left(\max_{|x|=n}d_x^+\right)^{1/n}=1,$$ 
then $\mathrm{gr}_+(G)=\mathrm{gr}(G)$.

 From the Nash-Williams criterion (\cite{LR-PY2016} Section 2.5), it follows that for any $G$ with $\mathrm{gr}_+(G)<\infty$, $\RW_\lambda$ is recurrent for $\lambda>\mathrm{gr}_+(G)$ and thus $\lambda_c(G)\le\mathrm{gr}_+(G)$.
If $G$ is spherically symmetric then $\lambda_c(G)=\mathrm{gr}_+(G)$ (\cite{LR-PY2016} Section 3.4, Exercise 3.11).

\medskip

An original motivation for introducing $\RW_\lambda$ by Berretti and Sokal \cite{BA-SA-1985} was to design a Monte-Carlo algorithm for self-avoiding walks. See \cite{LG-SA1988,SA-JM1989, RD1994} for refinements of this idea. Since the 1980s biased random walks and biased diffusions in disordered media have attracted much attention in mathematical and physics communities  due to their interesting phenomenology and similarities to concrete physical systems (\cite{BM-DD1983, DD1984, DD-SD1998, HS-BA2002}). 
 In the 1990s, Lyons (\cite{LR1990, LR1992, LR1995}), and Lyons, Pemantle and Peres (\cite{LR-PR-PY1996a, LR-PR-PY1996b}) made a fundamental advance in the study of $\RW_\lambda$'s. $\RW_\lambda$ has also received attention recently, see \cite{BG-HY-OS2013, AE2014, BG-FA-SV2014, HY-SZ2015} and the references therein. For a survey on biased random walks on random graphs 
see Ben Arous and Fribergh \cite{BG-FA2014}.

\medskip

This paper focuses on a specific properties of spectral radius of $\RW_\lambda$'s on non-random infinite graphs.
 {The uniform spanning forests of the network associated with $\RW_{\lambda}$ on the Euclidean lattices  are studied in a companion paper \cite{SSSWX2017b+}.}


\medskip

\medskip

Let us introduce some basic notation. Write
$$
p^{(n)}(x,\, y)
:=
p^{(n)}_\lambda(x,\, y)=\mathbb{P}_x(X_n=y),
$$

\noindent where $\mathbb{P}_x:=\mathbb{P}_x^{G}$ is the law of $\RW_\lambda$ starting at $x$. The Green function is given by
\[
\mathbb{G}(x,\, y\, |\, z)
:=
\mathbb{G}_\lambda(x,\, y\, | \, z)
=
\sum_{n=0}^\infty p^{(n)}(x,\, y)z^n,~ x,~y\in V(G),~z\in \mathbb{C},\ |z| < R_{\mathbb{G}} \, ,
\]

\noindent where $R_\mathbb{G}=R_\mathbb{G}(\lambda)=R_\mathbb{G}(\lambda,x,y)$ is its convergence radius. Note that
\[
R_\mathbb{G}
=
R_\mathbb{G}(\lambda)
=
\frac{1}{\limsup_{n\to \infty}\sqrt[n]{p^{(n)}(x,y)}}
\]

\noindent is independent of $x$, $y$ when $\RW_\lambda$ is irreducible, i.e., $\lambda>0$. When $\lambda=0$, $R_\mathbb{G}(0)=\infty$. Call
\[
\rho_\lambda
=
\rho(\lambda)
=
\frac{1}{R_\mathbb{G}}
=
\limsup_{n\to \infty} p^{(n)}(x,x)^{1/n}
=
\limsup_{n\to \infty} p^{(n)}(o,o)^{1/n}
\]

\noindent the spectral radius of $\RW_\lambda$.

%
%

\medskip

We are ready to state our first main result.

\begin{thm}\label{thm2.1}

 Let $G$ be a locally finite, connected infinite graph.
 
 \medskip

\noindent {\bf (i)} The spectral radius $\rho_\lambda$ is continuous in $\lambda\in
(0,\infty)$, {and} 
$\rho(\lambda_c) = 1$.

\smallskip

\noindent {\bf (ii)} If $\rho_{\lambda}$ is continuous at $0$, then there are no adjacent
vertices in $\partial B_G(n)$ for any $n\in\mathbb{N}$, and $d_v-d_v^{-}\ge 1$
for any vertex $v$. 

\smallskip

\noindent Conversely, on any infinite graph $G$, if for any $n\in\mathbb{N}$ there are no adjacent
vertices in $\partial B_G(n)$, and {if}
there exists
$\delta>0$ such that $d_v - d_v^- \ge \delta d_v$ for any vertex $v$, then $\rho_{\lambda}$ is
continuous at $0$.


\end{thm}


\medskip

Let $d\in\mathbb{N}$, $d\ge2$, and $\mathcal{G}_d$ denotes the set of all $d$-regular infinite connected graphs. 

\smallskip

\begin{thm}\label{thm2.2}

 Let $G\in\mathcal{G}_d$, and $\lambda\in \left(0, \, \lambda_c(\mathbb{T}_d) = d-1\right)$. 
 
 \smallskip

{\bf (i)} We have
 $$
 \rho_G(\lambda)
 \ge
 \rho_{\mathbb{T}_d}(\lambda)
 =
 \frac{2\sqrt{(d-1)\lambda}}{d-1+\lambda} \, .
 $$

{\bf (ii)} Assume $G$ is transitive. Then
 $$
 \rho_G(\lambda)
 =
 \rho_{\mathbb{T}_d}(\lambda)
 \ \mbox{if and only if}\
 G
  \mbox{ is isomorphic to } \mathbb{T}_d.
 $$
\end{thm}


In the case $\lambda=1$, Theorem \ref{thm2.2} follows from Kesten \cite[Theorem 2]{KH1959} (see \cite[p.~122 Corollary 11.7]{WW2000} and \cite[Theorem 6.11]{LR-PY2016}). {For this case ($\lambda=1$) our proof of Theorem \ref{thm2.2}  differs from the proofs  in \cite{KH1959}, \cite{WW2000} and \cite{LR-PY2016}. } 

The rest of the paper is organized as follows. 
{We prove Theorem~\ref{thm2.1} and \ref{thm2.2} in Section \ref{s:RW_general}. In Section \ref{sec3}, we focus on the spectral radius and the speed for $\RW_\lambda$'s on free product of graphs.
}

When emphasizing {that} a function $g(\cdot)$ depends on the underlying graph $G$, we will use $g_G(\cdot)$ or $g^G(\cdot)$ to replace $g(\cdot)$. \\

\section{Proofs of Theorem \ref{thm2.1} and Theorem \ref{thm2.2}}
\label{s:RW_general}

For any vertex set $A$, let
$$
\tau_A=\inf\{n\ge 0 \, |~X_n \in A\},
\qquad
\tau_A^+=\inf\{n\ge 1 \, |~X_n \in A\}.
$$

\noindent When $A=\{y\}$, write $\tau_y=\tau_{\{y\}}$, $\tau_y^{+}=\tau_{\{y\}}^{+}$. Put 
\begin{eqnarray}
 && f^{(n)}(x, \, y)
    :=
    f^{(n)}_\lambda(x, \, y)
    =
    \mathbb{P}_x (\tau_y^+=n),
    \label{fn}
    \\
 &&U(x, \, y\, |\, z)
    :=
    U(x, \, y \, | \, z)
    =
    \sum_{n=1}^\infty f^{(n)}(x,\, y)z^n,~x,~y\in V(G),
    \qquad
    z\in \mathbb{C},\ |z| <R_U ,
    \label{U}
\end{eqnarray}

\noindent where $R_U=R_U(\lambda)=R_U(\lambda,\, x,\, y)$ is the convergence radius of $U$, which is also independent of $x,y$ for $\lambda>0$. When $\lambda=0$, $R_U(0)=\infty$.

\subsection{Proof of Theorem \ref{thm2.1} part (i)}\label{sec2}




\begin{proof}
It suffices to verify that the convergence radius $R_{\mathbb{G}}(\lambda)$ is continuous in $\lambda\in (0, \, \infty)$. This is done in tree steps.

{\bf Step 1.} For any sequence $\{\lambda_k\}_{k\ge 1}\subset (0, \, \lambda_c(G)]$ converging to a limit $\lambda_0\in (0, \, \lambda_c(G)]$, we claim that
$$
\limsup_{k\to\infty} R_{\mathbb{G}}(\lambda_k)
\le
R_{\mathbb{G}}(\lambda_0).
$$
Assume
$$\limsup_{k\to\infty} R_{\mathbb{G}}(\lambda_k)
>
R_{\mathbb{G}}(\lambda_0)=z_*.
$$

\noindent then there exists a subsequence $\{\lambda_{n_k}\}_{k\ge 1}$, such that $a=\lim_{k\to \infty}R_{\mathbb{G}}(\lambda_{n_k})>z_*$.
For any $z_0\in (z_*, \, a)$, when $k$ is sufficiently large,
$$
U_{\lambda_{n_k}}(o, \, o \, | \, z_0)
=
\sum_{n={1}}^{\infty}f_{\lambda_{n_k}}^{(n)}(o,\, o)z_0^n
<1 ,
$$

\noindent because $\mathbb{G}_{\lambda_{n_k}}(o, \, o \, | \, z_0)=\sum_{n=0}^{\infty}\{U_{\lambda_{n_k}}(o, \, o \, | \, z_0)\}^n<\infty$. By Fatou's lemma,
\begin{equation}
  \label{(2.1)}
  1\ge \liminf_{k\to\infty}\sum_{n={1}}^{\infty}f_{\lambda_{n_k}}^{(n)}(o,o)z_0^n\ge
    \sum_{n={1}}^{\infty}\liminf_{k\to\infty}f_{\lambda_{n_k}}^{(n)}(o,o)z_0^n
=\sum_{n={1}}^{\infty}f_{\lambda_{0}}^{(n)}(o,o)z_0^n =U_{\lambda_0}(o, \, o \, | \, z_0).
\end{equation}

We now distinguish two possible cases. First case: $R_{\mathbb{G}}(\lambda_0)=R_U(\lambda_0)$. Since $z_0 > z_* = R_{\mathbb{G}}(\lambda_0)$, we would have $U_{\lambda_0}(o, \, o \, | \, z_0)=\infty$, leading to a contradiction. Second case: $R_{\mathbb{G}}(\lambda_0)<R_U(\lambda_0)$. For any $z>z_*$, $\infty= \mathbb{G}_{\lambda_0}(o, \, o \, | \, z)=\sum_{n=0}^{\infty}\{U_{\lambda_0}(o, \, o \, | \, z)\}^n$, so $U_{\lambda_0}(o, \, o \, | \, z) \ge 1$. Since $U_{\lambda_0}(o, \, o \, | \, z)$ is strictly increasing in $z\in [0, \, R_U(\lambda_0))$, this would again contradict (\ref{(2.1)}).

{\bf Step 2.} We prove in this step that $\liminf_{k\to\infty}
R_{\mathbb{G}}(\lambda_k) \ge R_{\mathbb{G}}(\lambda_0)$ for any sequence
$\{\lambda_k\}_{k\ge 1}$ converging to a limit $\lambda_0 \in (0, \, \infty)$.

For any $n\in\mathbb{Z}_{+}$, let
\begin{eqnarray*}
    \Pi_n
 &=& \{\mbox{paths}\ \gamma\ \mbox{in}\ G\ \mbox{staring and ending at}\ o\ \mbox{with length}\ n\},
    \\
    \mathbb{P}(\gamma, \, \lambda)
 &=& \prod_{i=0}^{n-1}p_{\lambda}(w_i,w_{i+1}),\ \gamma=w_0w_1\cdots w_n\in \Pi_n \, .
\end{eqnarray*}

Note that for $0 < \lambda_1 \le \lambda_2 < \infty$ and $v \sim u$ we
have
\begin{equation}
  \label{e:tfratio}
  \frac{\lambda_1}{\lambda_2}
  \le
  \frac{p_{\lambda_1}(v,\, u)}{p_{\lambda_2}(v, \, u)}
  \le
  \frac{\lambda_2}{\lambda_1} \, .
\end{equation}

\noindent Thus, for any $\delta > 0$, there is a constant $\varepsilon > 0$
such that $p_{\lambda}(v,\, u) \le (1+\delta) p_{\lambda_0}(v,\, u)$ for $\lambda \in (\lambda_0-\varepsilon, \, \lambda_0+\varepsilon)$. Consequently, we have $\mathbb{P}(\gamma, \, \lambda) \le (1+\delta)^n \mathbb{P}(\gamma, \, \lambda_0)$ for $\gamma \in \Pi_n$ and
$$
p^{(n)}_{\lambda}(o, \, o)
=
\sum_{\gamma \in \Pi_n} \mathbb{P}(\gamma, \, \lambda)
\le
\sum_{\gamma \in \Pi_n} (1+\delta)^n(\gamma, \, \lambda_0)
=
(1+\delta)^n p^{(n)}_{\lambda_0}(o, \, o).
$$

\noindent Therefore we have for $k$ large enough,
$$
\mathbb{G}_{\lambda_k}(o, \, o\, | \, z)
=
\sum_{n=0}^{\infty} p^{(n)}_{\lambda_k}(o,\, o) z^n
\le
\sum_{n=0}^{\infty} p^{(n)}_{\lambda_0}(o,\, o) \left( (1+\delta)z \right)^n
<
\infty,
$$

\noindent provided $(1+\delta)z < R_{\mathbb{G}}(\lambda_0)$.
Since $\delta$ is arbitrary, we have that $\liminf_{k\to\infty}
R_{\mathbb{G}}(\lambda_k) \ge R_{\mathbb{G}}(\lambda_0)$.

\textbf{Step 3.} It remains to prove $R_{\mathbb{G}}(\lambda_c) = 1$. Suppose
$R_{\mathbb{G}}(\lambda_c) > 1$, then for $\lambda>\lambda_c$ and $z>1$ with
$1<\frac{\lambda z}{\lambda_c} < R_{\mathbb{G}}(\lambda_c)$, we would have from
\eqref{e:tfratio} that
$$
\sum_{n=0}^{\infty} p^{(n)}_{\lambda}(o,\, o) z^n
\le
\sum_{n=0}^{\infty} p^{(n)}_{\lambda_c}(o,\, o) \left( \frac{\lambda z}{\lambda_c}
\right)^n
<
\infty.
$$

\noindent Then $R_{\mathbb{G}}(\lambda) > 1$. This contradicts to the fact that $\mathrm{RW}_{\lambda}$ is recurrent for $\lambda > \lambda_c$.


\bigskip


\subsection{Proof of Theorem \ref{thm2.1} part (ii)}
{We split the proof of (ii) into three steps. }

\medskip

 \noindent  {\bf Step 1.} For any given locally finite, connected infinite graph $G$, such that $\partial B_G(n_0)$ contains adjacent vertices for some $n_0$ we prove that $\rho_\lambda$ is not continuous at $0$.

Let $u$ and $v$ be adjacent vertices in $\partial B_G(n_0)$. Let $e=\{u, \, v\}$ and $x_0=o$. For $\RW_\lambda$ (with $\lambda>0$) to return to $o$, it suffices to walk first along a path $\gamma=x_0x_1\cdots x_{n_0}$ of length $n_0$   to a vertex $u\in \partial B_G(n_0)$ in $n_0$ steps, then walk $2n$ steps between $u$ and $v$, and finally returns to $o$ from $u$ along $\widetilde{\gamma}=x_{n_0}x_{n_0-1}\cdots x_1x_0$. Accordingly,
\begin{equation}
    p_{\lambda}^{(2n+2n_0)}(o, \, o)
    \ge
    \mathbb{P}(\gamma,\, \lambda) \, \mathbb{P}(\widetilde{\gamma}, \,\lambda)
    \left(\frac{1}{d_u+(\lambda-1)d_u^{-}}\right)^n
    \left(\frac{1}{d_v+(\lambda-1)d_v^{-}}\right)^n,
    \label{(2.2)}
\end{equation}

\noindent where for any $\lambda>0$,
$$
\mathbb{P}(\gamma, \, \lambda)
=
\prod_{i=0}^{n_0-1}p_{\lambda}(x_i,\, x_{i+1})>0,
\qquad
\mathbb{P}(\widetilde{\gamma}, \, \lambda)
=
\prod_{i=0}^{n_0-1}p_{\lambda}(x_{n_0-i}, \, x_{n_0-i-1})>0.
$$

\noindent So for any $\lambda>0$,
$$
\rho_\lambda
\ge
\limsup_{n\to\infty} \left\{p_{\lambda}^{(2n+2n_0)}(o, \, o)\right\}^{\frac{1}{2n+2n_0}}
\ge
\frac{1}{\{[d_u+(\lambda-1)d_u^{-}] \, [d_v+(\lambda-1)d_v^{-}] \}^{1/2}}>0.
$$

\noindent Letting $0<\lambda\to 0$, we immediately get
$$
\liminf_{\lambda\to 0+}\rho_\lambda
\ge
\frac{1}{[(d_u-d_u^{-})(d_v-d_v^{-})]^{1/2}}
>
0
=
\rho_0.
$$

\noindent \textbf{Step 2.} Assume that there is a vertex $v$ such that $d_v - d_v^- = 0$. Let $u$ be a vertex adjacent to $v$. Let $\gamma$ be a path from $o$
to $u$ of length $n_0$, and denote by $\tilde{\gamma}$ the reverse path. Similar to the arguments in the previous step, we have for any $n$,
$$
p^{(2n_0 + 2n)}_{\lambda}(o, \, o)
\ge
\mathbb{P}(\gamma, \, \lambda) \mathbb{P}(\tilde{\gamma}, \, \lambda) \Big( \frac{1}{d_u +
    (\lambda-1)d_u^-} \Big)^n \Big( \frac{1}{d_v} \Big)^n.
$$

\noindent Then for any $\lambda > 0$,
$$
\rho_{\lambda}
\ge
\Big( \frac{1}{d_v(d_u+(\lambda-1)d_u^-)} \Big)^{1/2}
>
0.
$$

\noindent Hence $\rho_{\lambda}$ is not continuous at $0$.

\noindent {\bf Step 3.} Assume that there are no adjacent
vertices in $\partial B_G(n)$ for any $n\in\mathbb{N}$, and there exists
$\delta>0$ such that $d_v - d_v^- \ge \delta d_v$ for any vertex $v$.
Then for any $\lambda>0$ and the $\RW_\lambda\ (X_n)_{n=0}^\infty$, the following holds almost surely:
\begin{eqnarray}\label{(2.3)}
    |X_{n+1}| - |X_n| \in \{+1,-1\},
    \qquad
    \forall n\in\mathbb{Z}_{+}.
\end{eqnarray}

\noindent When $X_0=o$, the walk $(X_n)_{n=0}^\infty$ takes an even number (say, $2m$, for some $m\ge 1$) of steps to return to $o$: Among these $2m$ steps, $m$ steps are upward and the other $m$ steps are downward.

When $v\sim u$ and $|u| = |v| -1$, we have
$$
p_{\lambda}(v,\, u)
=
\frac{\lambda}{d_v+(\lambda-1)d_v^{-}}
\le
\frac{\lambda}{d_v-d_v^{-}}
\le
\lambda d_v^{-1} \delta^{-1},
\qquad
\lambda>0.
$$

\noindent When $v\sim u$ and $|u| = |v|+1$, we have $p_{\lambda}(v,\, u) \le
d_v^{-1} \delta^{-1}$. Hence for any path $\gamma=w_0w_1\cdots w_{2n}\in\Pi_{2n}$,
$$
\mathbb{P}(\gamma, \, \lambda)
=
\prod_{i=0}^{2n-1}p_{\lambda}(w_i, \, w_{i+1})
\le
\lambda^n \delta^{-2n} \mathbb{P}(\gamma,1),
\qquad
\lambda>0,
$$

\noindent which implies that for any $\lambda>0$,
$$
p_\lambda^{(2n)}(o, \, o)
=
\sum_{\gamma\in\Pi_{2n}}\mathbb{P}(\gamma, \, \lambda)
\le
\lambda^n \delta^{-2n} \sum_{\gamma\in\Pi_{2n}} \mathbb{P}(\gamma,1)
\le
\lambda^n \delta^{-2n} p^{(2n)}_1(o,o).
$$

\noindent Hence
$$
\rho_\lambda
=
\limsup_{n\to\infty}\left\{p_\lambda^{(2n)}(o,\, o)\right\}^{\frac{1}{2n}}
\le
\delta^{-1} \rho_1 \lambda^{1/2},
$$

\noindent proving that $\lim_{\lambda\to 0+}\rho_\lambda=0=\rho_0$.
\end{proof}

\subsection{Proof of Theorem \ref{thm2.2}}

We start with the lemma, which {will be used in the proof of Theorem~\ref{thm2.2}.  For readers' convenience we provide the proof in Appendix~\ref{s:pfthm2.2}.} 

\begin{lem}\label{lem2.6}
For the $d$-regular tree $\mathbb{T}_d$, the following holds:
$$
\theta_{\mathbb{T}_d}(\lambda)
=
\frac{\lambda}{d-1},
\qquad
\rho_{\mathbb{T}_d}(\lambda)
=
\frac{2\sqrt{(d-1)\lambda}}{d-1+\lambda},
\qquad
\lambda\in [0, \, \lambda_c(\mathbb{T}_d)]=[0,\, d-1],
$$
and for $\lambda\in (0, \, \infty)$ and $n\to \infty$,
\begin{equation}
    f^{(2n)}_\lambda(o, \, o)
    \sim
    \frac{1}{\sqrt{\pi}} \left( \frac{2\sqrt{(d-1)\lambda}}{d-1+\lambda}\right)^{\!\! 2n} n^{-3/2}.
    \label{f_regular_tree}
\end{equation}
Moreover,
\begin{equation}
    p^{(2n)}_\lambda(o,o)
    \sim
    \begin{cases}
          \frac{(d-1-\lambda)^2}{16(\pi\lambda)^{1/2}(d-1)^{3/2}}\rho_{\mathbb{T}_d}(\lambda)^{2n}n^{-3/2}&\hbox{if}\ \lambda\in (0, \, d-1),
      \\
    \frac{1}{\sqrt{\pi n}}&\hbox{if}\ \lambda=d-1. 
    \end{cases}
    \label{p_regular_tree}
\end{equation}
\end{lem}



\medskip

\noindent {Now we are ready to give the proof of Theorem~\ref{thm2.2}.}

\begin{proof}[Proof of Theorem~\ref{thm2.2}]
  (i) 
  Fix $\lambda \in (0, \, \lambda_c(\mathbb{T}_d)]$. Define $g=g_\lambda:\ \mathbb{Z}_{+}\to \mathbb{R}$ by
$$
g(n)
=
g_\lambda(n)
:=
\left(1+\frac{d-1-\lambda}{d-1+\lambda}n\right)((d-1)/\lambda)^{-n/2},
$$

\noindent and $f=f_\lambda:\ G\to\mathbb{R}$ by
\begin{eqnarray}
\label{(2.7)}
    f(x):=f_\lambda(x)=g(|x|),
    \qquad
    \forall x\in V(G).
\end{eqnarray}

\noindent Clearly, $g$ is non-increasing on $\mathbb{Z}_{+}$.

Recall $p(x,\, y)$ from (\ref{(1.1)}). For any $h:\ G\to\mathbb{R}$, let
\begin{equation}
    \label{(2.8)}
    Ph(x)
    :=
    \sum_{y\sim x} p(x,\, y) h(y),
    \qquad
    x\in V(G).
\end{equation}

\noindent Then $Pf(o)=\rho_{\mathbb{T}_d}(\lambda)f(o)$, and for $x\not=o$,
\begin{align*}
    Pf(x)
 &=\frac{d_x^+g(|x|  +1)+d_x^0g(|x|)+\lambda d_x^-g(|x|-1)}{d_x^++d_x^0+\lambda d_x^{-}}
    \\
    \,
    \\
 &\ge \frac{(d_x^+ + d_x^0) g(|x|+1)+ \lambda d_x^-g(|x|-1)}{d_x^++d_x^0+\lambda d_x^{-}} .
\end{align*}

\noindent Since $g(|x|-1) \ge g(|x|+1)$ and $d_x^- \ge 1$ (so $d_x^+ + d_x^0 \le d-1$), this leads to:
\begin{equation}
    Pf(x)
    \ge
    \frac{(d-1)g(|x|+1)+ \lambda g(|x|-1)}{d-1+\lambda}
    =
    \rho_{\mathbb{T}_d}(\lambda) f(x),
    \qquad
    x\not=o.
    \label{Pf>}
\end{equation}

\noindent For further use, we notice that for $x\not=o$, if $Pf(x) = \rho_{\mathbb{T}_d}(\lambda) f(x)$, then $d_x^- =1$, $d_x^0=0$ and $d_x^+ = d-1$.

For any $n\in\mathbb{N}$, put $f_n := f\, I_{B_G(n)}$. For $x\in B_G(n)$,
$$
Pf_n(x)
=
Pf(x)-\frac{d_x^+g(n+1)}{d_x^++d_x^0+\lambda d_x^-}I_{\{|x| =n\}}.
$$
Define $\mu$ as follows: $\mu(o)=d_o$ and $\mu(x) = (d_x^{+}+d_x^0+\lambda d_x^-)\lambda^{-|x| }$ for $x\not=o$. Let $M_n := |\partial B_G(n)|$ as before. Denote by $(\, \cdot\, , \, \cdot\, )$ the inner product of $L^2(G,\mu)$. Then
$$
(Pf_n, \, f_n)
=
\sum_{x\in B_G(n)}Pf(x) f(x) \mu(x)
-
\sum_{x\in\partial B_G(n)} \frac{d_x^+g(n+1)}{d_x^++d_x^0+\lambda d_x^-}f(x)\mu(x) \, .
$$

\noindent For the sum $\sum_{x\in B_G(n)}$ on the right-hand side, we observe that by \eqref{Pf>}, for $x\in B_G(n)$, $Pf(x) \le \rho_{\mathbb{T}_d} (\lambda) f(x) = \rho_{\mathbb{T}_d} (\lambda) f_n(x)$. For the sum $\sum_{x\in\partial B_G(n)}$, we note that for $x\in\partial B_G(n)$, since $d_x^+ \le d-1$ and $f(x) = g(n)$, we have $\frac{\mu(x)}{d_x^++d_x^0+\lambda d_x^-} = \lambda^{-n}$. Accordingly,
\begin{align*}
    (Pf_n, \, f_n)
 &\ge \rho_{\mathbb{T}_d}(\lambda)(f_n,\, f_n)
    -
    (d-1)M_n \, g(n) g(n+1) \, \lambda^{-n}
 &\ge \rho_{\mathbb{T}_d}(\lambda)(f_n,\, f_n)
    -
    (d-1)M_n \, g(n)^2\lambda^{-n} ,
\end{align*}

\noindent which implies that
$$
\rho_G(\lambda)
=
\sup_{h\in L^2(G,\mu)\setminus\{0\}} \frac{(Ph,\, h)}{(h,\, h)}
\ge
\frac{(Pf_n,\, f_n)}{(f_n,\, f_n)}
\ge
\rho_{\mathbb{T}_d}(\lambda)
-
\frac{(d-1)M_n \, g(n)^2\lambda^{-n}}{(f_n,\, f_n)} \, .
$$

\noindent Observe that
\begin{align*}
    (f_n, \, f_n)
 &=\sum_{k=0}^n \sum_{x\in\partial B_G(k)} g(k)^2 \mu(x)
    =
    \sum_{k=0}^n\sum_{x\in\partial B_G(k)} g(k)^2\, (d_x^++d_x^0+\lambda d_x^-) \, \lambda^{-|x|}
    \\
 &\ge (\lambda\wedge 1) d\, \sum_{k=0}^n M_k\, g(k)^2 \, \lambda^{-k}.
\end{align*}

\noindent Hence
$$
\rho_G(\lambda)
\ge
\rho_{\mathbb{T}_d}(\lambda)
-
\frac{d-1}{d(\lambda\wedge 1)} \frac{M_n\, g(n)^2\, \lambda^{-n}}{\sum_{k=0}^n M_k\, g(k)^2\, \lambda^{-k}}.
$$

\noindent It remains to prove that
$$
\lim_{n\to\infty} \frac{M_n \, g(n)^2 \, \lambda^{-n}}{\sum_{k=0}^n M_k \, g(k)^2 \, \lambda^{-k}}=0.
$$

For $k\le n$,
$$
M_n \, g(n)^2\, \lambda^{-n}
\le
M_k \, (d-1)^{n-k}g(n)^2\, \lambda^{-n}
=
M_k \, g(k)^2\, \lambda^{-k}\left(\frac{(d-1-\lambda)n+d-1+\lambda}{(d-1-\lambda)k+d-1+\lambda}\right)^2,
$$

\noindent which implies that
$$
\frac{\sum_{k=0}^n M_k \, g(k)^2 \, \lambda^{-k}}{M_n \, g(n)^2 \, \lambda^{-n}}
\ge
\sum_{k=0}^n \left(\frac{(d-1-\lambda)k+d-1+\lambda}{(d-1-\lambda)n+d-1+\lambda}\right)^2 .
$$

\noindent Since $\lambda \le d-1$, the sum on the right-hand side goes to infinity as $n\to \infty$.

\bigskip

(ii)
For $d=2$, $\mathcal{G}_d=\{\mathbb{T}_2\}$, the result holds trivially. So we assume $d\ge 3$. It suffices to prove that for any transitive $G\in\mathcal{G}_d$ with the minimal cycle length $\ell\ge 3$,
\begin{equation}
    \label{(2.9)}
    \rho_G(\lambda)>\rho_{\mathbb{T}_d}(\lambda),
    \qquad
    \forall\lambda\in (0, \, \lambda_c(\mathbb{T}_d)).
\end{equation}

{\bf Step 1.} $\lambda_c(G)<\lambda_c(\mathbb{T}_d)=d-1$.

Let $\Gamma_{d,\ell} := \langle a_1, \ldots, a_{d-2}, \, b\, | \, a_i^2=1,\ b^\ell =1\rangle$
be a finitely-presented group with generating set $S= \{a_1, \ldots, a_{d-2},\, b, \, b^{-1} \}$, and $\mathbb{X}_{d,\ell} := (\mathbb{Z}_2*\cdots *\mathbb{Z}_2) \, (d-2 \hbox{ \rm folds}\, )*\mathbb{Z}_\ell$
the corresponding Cayley graph; then the transitive graph $G$ is covered by $\mathbb{X}_{d,\ell}$  (see Theorem 11.6 of \cite{WW2000}).  From this result, we obtain
$$
\lambda_c(G)
=
\mathrm{gr}(G)
\le
\mathrm{gr}(\mathbb{X}_{d,\ell}).
$$

For $z\ge 0$, define
\begin{eqnarray*}
 &&k_\ell(z)=\left\{\begin{array}{ll}
  2z+2z^2+\cdots+2z^{\frac{\ell-1}{2}},\ &\ \mbox{if}\ \ell\ \mbox{is odd},\\
  2z+2z^2+\cdots+2z^{\frac{\ell-2}{2}}+z^{\frac{\ell}{2}}, &\ \mbox{if}\ \ell\ \mbox{is even};
  \end{array}
 \right.\\
 &&h_\ell(z)=\frac{(d-2)z}{1+z}+\frac{k_\ell(z)}{1+k_\ell(z)}.
\end{eqnarray*}

\noindent Then $\mathrm{gr}(\mathbb{X}_{d,\ell})=\frac{1}{z_*}$ where $z_*$ is the unique positive number satisfying $h_\ell(z_*)=1$ (see \cite{CE-GL-MS2012} p.~28; it will also be recalled in more details in (\ref{(3.1)}) below). Since $j_\ell := \frac{k_\ell (\frac{1}{d-1})}{1+k_\ell (\frac{1}{d-1})}$ is strictly increasing in $\ell$, and $\lim_{r\to\infty} j_r = \frac2d$, we have $j_\ell<\frac2d$, which implies $h_\ell (\frac{1}{d-1})<1$. Notice that $h_\ell(z)$ is strictly increasing in $z\ge 0$. So $z_*>\frac{1}{d-1}$ and $\mathrm{gr}(\mathbb{X}_{d,\ell}) = \frac{1}{z_*} <d-1$, which implies $\lambda_c(G) < d-1$.

{\bf Step 2.} Fix $\lambda\in (0,\, d-1)$. Let as before $\mu(o) := d_o$ and $\mu(x) := (d_x^{+}+d_x^0+\lambda d_x^-)\lambda^{-|x|}$ if $x\not=o$. Let $f:\ G\to\mathbb{R}$ be the function defined in (\ref{(2.7)}). Then $f\in L^2(G,\, \mu)$.

Since $G$ is transitive, $\lambda_c (G) = \mathrm{gr} (G) = \lim_{n\to \infty} M_n^{1/n}$. By Step 1, for any $\varepsilon \in (0, \, d-1-\lambda_c(G))$, there is a constant $c_\varepsilon > 0$ such that
$$
M_n
\le
c_\varepsilon\, (\lambda_c(G)+\varepsilon)^n,
\qquad
\forall n\ge 0.
$$

\noindent Thus
\begin{align*}
    \sum_{x\in V(G)} f^2(x) \, \mu(x)
 &= \sum_{x\in V(G)} \left(1+\frac{d-1-\lambda}{d-1+\lambda} |x|\right)^2 \left(\frac{\lambda}{d-1}\right)^{|x|} \left( d_x^++d_x^0+\lambda d_x^{-}\right) \lambda^{-|x|}
    \\
 &\le (\lambda\vee 1)d\, \sum_{n=0}^\infty M_n\left(1+\frac{d-1-\lambda}{d-1+\lambda}n\right)^2\left(\frac{1}{d-1}\right)^n\\
     &\le (\lambda\vee 1)dc_\varepsilon \, \sum_{n=0}^\infty
         \left(\frac{\lambda_c(G)+\varepsilon }{d-1}\right)^n\left(1+\frac{d-1-\lambda}{d-1+\lambda}n\right)^2\\
     &<\infty.
\end{align*}

{\bf Step 3.} (\ref{(2.9)}) is true.

Let $\lambda\in (0, \, \lambda_c(\mathbb{T}_d))$. We have noticed in the proof of 
{(i)} 
that $Pf (o)=\rho_{\mathbb{T}_d}(\lambda)f (o)$ and that for $x\not=o$,
$$
Pf(x)\ge \rho_{\mathbb{T}_d}(\lambda)f(x),
\ \mbox{and}\ ``="\ \mbox{implies}\ d_x^{-}=1, \, d_x^0=0, \, d_x^+=d-1.$$

\noindent Since the transitive $G$ has the minimal cycle length $\ell\ge 3$, we cannot have $d_x^-=1$, $d_x^0=0$, $d_x^+=d-1$ for any $x\in V(G)\setminus\{o\}$. Note that $f(\cdot)$ and $\mu(\cdot)$ are strictly positive on $G$. Hence
$$
(Pf, \, f)
=
\sum_{x\in V(G)}Pf(x) f(x) \mu(x)
>
\sum_{x\in V(G)}\rho_{\mathbb{T}_d}(\lambda)f^2(x)\mu(x)
=
\rho_{\mathbb{T}_d}(\lambda)(f, \, f).
$$

\noindent By Step 2, $f\in L^2(G, \, \mu)$, which implies that
$$
\rho_G(\lambda)
=
\sup_{h\in L^2(G,\mu)\setminus\{0\}}\frac{(Ph,\, h)}{(h,\, h)}
\ge
\frac{(Pf,\, f)}{(f,\, f)}
>
\rho_{\mathbb{T}_d}(\lambda),
$$

\noindent proving (\ref{(2.9)}). 
\end{proof}

\bigskip

 Since for some $G\in\mathcal{G}_d$ 
that are not trees, one may have $\mathrm{gr}(G)=d-1$, in general it is not true that $f\in L^2(G, \, \mu)$ for $\lambda\in (0, \, d-1)$. However, for any transitive graph $G\in\mathcal{G}_d$ that is not isomorphic to $\mathbb{T}_d$, we have $\mathrm{gr}(G)<d-1$, which ensures $f\in L^2(G,\mu)$ in the proof of Theorem~\ref{thm2.2} (ii).

\section{Biased random walks on free product of graphs}\label{sec3}

The study of random processes on free products of graphs goes back at least to Teh and Gan \cite{TG70}, Zno\v{i}ko \cite{Z75} and Lyndon and Schupp \cite{LS77}. The recursive structure of such graphs often makes it possible to do explicit computations, leading to close-form analytical formulas. For simple random walks on free products of graphs, the spectral radius (see, for example, Woess \cite{WW2000} p. 101-110) and the critical percolation probability (\v{S}pakulov\'{a} \cite{S09}) are known. When $\lambda\not=1$, the biased random walks are not transitive any more, making computations more delicate. In this section, we determine the spectral radius and the speed of the biased random walk on the free product of two complete graphs.

Let $r\in\mathbb{N}, \; r \ge 2$. Write $\mathcal{I}=\{1,\ldots,r\}$. Let $\{G_i=(V_i,\, E_i, \, o_i)\}_{i\in\mathcal{I}}$ be a family of connected finite rooted graphs with vertex sets $V_i$, edge sets $E_i$ and roots $o_i$. Call a copy of $G_i$ an $i$-cell. Assume that each $|V_i|\ge 2$ for all $i$, and that all $V_i$'s are disjoint. Put
$$
V_i^\times
:=
V_i\setminus\{o_i\},
\qquad \mathrm{and}\qquad
\langle x\rangle
:=
i,
\qquad \mathrm{if} \ x\in V_i^\times,\ i\in\mathcal{I}.
$$
Define
$$
V
:=
V_1*\cdots *V_r
=
\Big\{ x_1x_2\cdots x_n\, \Big| \, n\in\mathbb{N}, \, x_i\in \bigcup_{j\in\mathcal{I}} V_j^\times, \ \langle x_i\rangle \not=\langle x_{i+1}\rangle \Big\} \cup\{o\}.
$$

\noindent We can also view $V$ as the set of words over the alphabet $\bigcup_{j\in\mathcal{I}}V_j^\times$ without two consecutive letters from the same $V_j^\times$, with $o$ denoting the empty word in $V$. Let
$$
\langle x_1\cdots x_n\rangle
:=
\langle x_n\rangle,\ \forall x_1\cdots x_n\in V;
\qquad
\langle o\rangle := 0.
$$

\noindent For any pair of words $x=x_1\cdots x_m$ and $y=y_1\cdots y_n\in V$ with $\langle x_m\rangle\not=\langle y_1\rangle$, the concatenation $xy$ of $x$ and $y$ is an element of $V$. In particular, $xo=ox=x$. When $\langle x\rangle\not=i\in\mathcal{I}$, we set $xo_i=o_ix=x$.

Define the set $E$ of edges on $V$ as follows: If $x$, $y\in V_i$ with $i\in\mathcal{I}$ and $x\sim y$, then
$$
wx\sim wy
\ \mbox{for any}\
w\in V
\ \mbox{with}\
\langle w\rangle\not=i.
$$

\noindent Then $G=(V, \, E,\, o)$ is the free product of the graphs $G_1$, $\ldots$, $G_r$, denoted by
$$
G=G_1*G_2*\cdots*G_r.
$$

\noindent {By \cite[Theorem 10.10]{WW2000}, $G$ is nonamenable if $r\ge 3$ or if $\max_{i\in \mathcal{I}} |V_i| \ge 3$.}

Let
$$
\partial B_{G_i}(n)
:=
\{x\in V_i:\ |x| =n\},
\qquad
\psi_i(z)
:=
\sum_{n\ge 1}|\partial B_{G_i}(n)|z^n, \ z\ge 0.
$$

\noindent {F}rom \cite[Lemma 4.15]{CE-GL-MS2012}, we have
\begin{equation}
    \label{(3.1)}
    \mathrm{gr}(G)
    =
    \frac{1}{z_*},
    \ \mbox{where $z_*$ is the unique postive number satisfying}\
    \sum_{i=1}^r\frac{\psi_i(z_*)}{1+\psi_i(z_*)}=1.
\end{equation}

Let $m_1$ and $m_2$ be positive integers such that $m_1 m_2 \geq 2$, and $K_{m_i + 1}$ the complete graph on $m_i + 1$ vertices (for $i=1$ and $2$).  We observe that by \eqref{(3.1)}, $\lambda_c (G) = \sqrt{m_1 m_2}$ when $G = K_{m_1 + 1} * K_{m_2 + 1}$.

\begin{thm}\label{thm3.1}
  Let $G:= K_{m_1 + 1} * K_{m_2 + 1}$ and $\lambda \in (0,\, \lambda_c(G))$. Let $m=m_1+m_2$. For $\RW_{\lambda}$ on $G$, 
  the following hold:
  \begin{enumerate}[(i)]
  \item The speed exists and equals 
$$
\mathcal{S}(\lambda)
=
\frac{2(m_1m_2-\lambda^2)}{(2\lambda+m)(\lambda+m-1)}. 
$$
In particular, $\mathcal{S}(\lambda)>0$ is smooth and strictly decreasing on $(0, \, \lambda_c(G))$. 
\item $\RW_{\lambda}$ has the non-Liouville property, namely, $\RW_{\lambda}$ has a non-constant bounded harmonic function.
\item The spectral radius
  \begin{align*}
  \rho(\lambda)
  =
  \frac{m-2 + [ (m_1 - m_2)^2 + 4 \lambda ( \sqrt{m_1} + \sqrt{m_2} )^2 ]^{1/2}}{2(m+\lambda-1)}.
  \end{align*}
  In particular, 
  $\lambda \mapsto \rho(\lambda)\in (0,\, 1)$ is strictly increasing on $(0,\, \lambda_c(G))$.
Moreover, for some constant $c>0$,
\[p^{(n)}_\lambda (o,o)\sim c \, \rho(\lambda)^n n^{-3/2} \quad  \text{as } n \to \infty. \] 
  \end{enumerate}
\end{thm}

\medskip

\noindent The proof of  Theorem \ref{thm3.1} is presented in Section~\ref{pfthm3.1}.

\subsection{Spectral radius for free product of complete graphs}
\label{subs:free_product_complete_graphs}

Let $r \geq 2$ and $m_i \geq 1$, $1 \le i \le r$. Let $G$ be the free 
product of the complete graphs $K_{m_i+1}$ with $m_i+1$ vertices. 
Let $z_*$ denote the unique positive number satisfying
\begin{equation}
    \label{(3.4)}
    \sum_{i=1}^r \frac{m_iz_*}{1+m_iz_*} = 1.
\end{equation}

\noindent By (\ref{(3.1)}),
\begin{equation}
    \label{(3.5)}
    \lambda_c(G)=\mathrm{gr}(G)=\frac{1}{z_*}.
\end{equation}

\noindent Write $m:= \sum_{i=1}^r m_i$.  The transition probability of $\RW_\lambda$ from $v$ to an adjacent vertex $u$ is
$$
p(v,\, u)
=
\begin{cases}
      \frac1m &{\rm if}\ v=o,\\
      \frac{\lambda}{m+\lambda-1} &{\rm if}\ u\in \partial B_G(|v|-1)\ \mbox{and}\ v\neq o,\\
     \frac{1}{m+\lambda-1} &{\rm otherwise}.
\end{cases}
$$



\begin{thm}
 \label{thm3.2}

 For $\lambda\in [0, \, \lambda_c(G))$, we have $\rho(\lambda)<1$. Moreover,
\begin{equation}
    \label{(3.6)}
    \rho(0+)
    =
    \frac{\max_{1\le i\le r}(m_i-1)}{m-1}.
\end{equation}

\end{thm}

\noindent {\it Proof.} {\bf Step 1.} Recall $U(o, \, o\, |\, z)$ and $R_U$ from \eqref{U}. For $z\in (-R_U, \, R_U)$,
\begin{equation}
    U(o, \, o\, |\, z)
    =
    \sum_{i=1}^r \frac{- (\phi_i(z)-mU(o, \, o\, |\, z))}{2m}
    +
    \sum_{i=1}^r\frac{[(\phi_i(z)- mU(o, \, o\, |\, z))^2+4\lambda m_iz^2]^{1/2}}{2m},
    \label{(3.7)}
\end{equation}

\noindent where $\phi_i(z) := m-1+\lambda-(m_i-1)z$.

To this end, let $\tau_o^+ := \inf\{ n\ge 1 \, | \, X_n =o\}$ as before, and for $i=1$, $2$, $\ldots$, $r$, let $f_i^{(n)}(o,\, o) := \mathbb{P}_o (\tau_o^+ =n, \ \langle X_1\rangle = i)$. Define
$$
U_i(o, \, o\, |\, z)
:=
\sum_{n=1}^\infty f_i^{(n)}(x,\, y)\, z^n,
\qquad
z\ge 0.
$$

\noindent Then
$$
U(o, \, o\, |\, z)
=
\sum_{i=1}^r U_i(o, \, o\, |\, z) ,
\qquad
z\ge 0.
$$

\noindent Note the tree-like structure of $G$. When the event $\{\tau_o^+ =n, \ \langle X_1\rangle = i\}$ occurs, $\RW_\lambda$ must visit an edge in $i$-cell
attached at $o$ at step 1 and return to $o$ the first time by an edge in the same $i$-cell at step $n$. Each vertex of the $i$-cell is attached to a certain $j$-cell (with $j\not= i$). {F}rom the spherical symmetry of each $K_{m_i+1}$, we obtain
$$
U_i(o, \, o\, |\, z)
=
\frac{m_i}{m}z \frac{\lambda}{m+\lambda-1}z
\sum_{n=0}^\infty
\left(M_1(z) + M_2(z) + \cdots + \widetilde{M}_i(z) + M_{i+1}(z) +\cdots + M_r(z)\right)^n ,
$$

\noindent where, for $j\not=i$,
$$
M_j(z)
:=
\sum_{n=1}^\infty \mathbb{P}_x [\tau_x^+=n,\ \langle X_1\rangle=j] \, z^n,
\qquad
x\in V(G),\ \langle x\rangle=i,\ |x| =1,
$$

\noindent which does not depend on $\langle x\rangle=i$, and
$$
\widetilde{M}_i(z)
:=
\mathbb{P}_x [ \langle X_1\rangle =i] \, z
=
\frac{m_i-1}{m+\lambda-1}z,
\qquad
x\in V(G), \ \langle x\rangle=i,\ |x| =1.
$$

\noindent By the similarity structure of $G$,
$$
M_j(z)
=
\Big( \frac{m_j}{m+\lambda-1} \Big/ \frac{m_j}{m}\Big) \, U_j(o, \, o\, |\, z)
=
\frac{m}{m+\lambda-1}\, U_j(o, \, o\, |\, z).
$$

\noindent So when $|z|<R_U$ (where $R_U$ denotes as before the convergence radius of $U$),
\begin{align*}
    U_i(o, \, o\, |\, z)
 &= \frac{\lambda m_i}{m(m+\lambda-1)}\, z^2 \,
    \frac{1}{1-[\frac{m}{m+\lambda-1} \sum_{j=1}^r U_j(o, \, o\, |\, z) - \frac{m}{m+\lambda-1} U_i(o, \, o\, |\, z) + \widetilde{M}_i(z)]}
    \\
 &= \frac{\lambda m_i}{m(m+\lambda-1)}\, z^2 \,
    \frac{1}{1-[\frac{m}{m+\lambda-1} U(o, \, o\, |\, z) - \frac{m}{m+\lambda-1} U_i(o, \, o\, |\, z) + \widetilde{M}_i(z)]} .
\end{align*}

\noindent Since $\widetilde{M}_i(z) = \frac{m}{m+\lambda-1}\, U_j(o, \, o\, |\, z)$, this yields, with the notation $\phi_i(z) := m-1+\lambda-(m_i-1)z$,
$$
U_i(o, \, o\, |\, z)
=
\frac{- (\phi_i(z) - mU(o, \, o\, |\, z))}{2m}
+
\frac{[(\phi_i(z) - mU(o, \, o\, |\, z))^2 + 4\lambda m_iz^2]^{1/2}}{2m},
$$

\noindent which implies (\ref{(3.7)}).

{\bf Step 2.} For any $0<\lambda<\lambda_c(G)$, $\mathbb{G}(o,\, o\, |\, R_{\mathbb{G}})<\infty$, $U(o,\, o\, | \, R_{\mathbb{G}})<1$, and $R_{\mathbb{G}}=R_U$.

Note that 
$R_{\mathbb{G}}\le R_U$, and that for $|z|< R_\mathbb{G}$, $|U(o, \, o\, |\, z)|<1$, $\mathbb{G}(o, \, o\, |\, z) = \frac{1}{1-U(o, \, o\, |\, z)}$. So $U(o,\, o\, |\, R_\mathbb{G}) = \lim_{z\uparrow R_\mathbb{G}} U(o, \, o\, |\, z) \le 1$.

Recall Pringsheim's Theorem: For $f(z)=\sum_{n=0}^\infty a_n z^n$ with $a_n\ge 0$, its convergence radius is the smallest positive singularity point of $f(z)$. As such, the smallest positive singularity point $R_{\mathbb{G}}$ of $\mathbb{G}(o, \, o\, |\, z)$ is either the smallest positive number $z_1$ with $U(o, \, o\, |\, z_1)=1$ if exists, or the convergence radius $R_U$ for $U(o, \, o\, |\, z)$. Since $U(o, \, o\, |\, z)$ is strictly increasing in $z\ge  0$, and $z_1$ is the unique positive number satisfying $U(o, \, o\, |\, z)=1$ if exists, it remains to prove that $U(o,\, o\, | \, R_{\mathbb{G}}) <1$ (which implies $\mathbb{G}(o,\, o\, | \, R_{\mathbb{G}}) < \infty$ and $R_{\mathbb{G}}=R_U$).

Assume this were note true; so $U(o,\, o\, | \, R_{\mathbb{G}}) = 1$. {\it{We exclude the trivial case where $m_i=1$ for $1\le i\le r$}} (in which case the result holds trivially; see the proof of Lemma \ref{lem2.6}). Note that $R_\mathbb{G}\ge 1$. If $R_\mathbb{G}=1$, then $U(o,\, o\, | \, R_{\mathbb{G}}) = U(o,\, o\, |\, 1)<1$ due to transience. So we assume $R_\mathbb{G}>1$.

By (\ref{(3.7)}),
\begin{align}
    1
 &= \sum_{i=1}^r\frac{-((\lambda-1) - (m_i-1)R_\mathbb{G}) +[ ( (\lambda-1) - (m_i-1) R_\mathbb{G} )^2 + 4\lambda m_iR_\mathbb{G}^2]^{1/2}}{2m}
    \nonumber
    \\
&=\sum_{i=1}^r\frac{(1-\lambda)+ (m_i-1) R_\mathbb{G} +[((1-\lambda) + (m_i-1) R_\mathbb{G})^2 + 4\lambda m_iR_\mathbb{G}^2]^{1/2}}{2m}.
    \label{pf:thm3.14}
\end{align}

\noindent We deduce a contradiction by distinguishing two possible cases.

{\bf Case 1.} $0< \lambda \le 1$. For any $1\le i\le r$,
$$
(1-\lambda)+(m_i-1)R_\mathbb{G}
\ge
(1-\lambda)+(m_i-1) \ge 0,
$$

\noindent  and the inequality is strict for at least one $i$. Thus by \eqref{pf:thm3.14},
\begin{align*}
    1
 &> \sum_{i=1}^r \frac{(1-\lambda) + (m_i-1) + [((1-\lambda) + (m_i-1))^2 + 4\lambda m_i]^{1/2}}{2m}
    \\
 &= \sum_{i=1}^r\frac{(m_i-\lambda)+(m_i+\lambda)}{2m}
    =
    1,
\end{align*}

\noindent which leads to a contradiction. Consequently,  in this case $U(o, \, o\, |\, R_\mathbb{G})<1$.

{\bf Case 2.} $1<\lambda<\lambda_c(G)$. Write
$$
[\lambda-1-(m_i-1)R_\mathbb{G}]^2 + 4\lambda m_i R_\mathbb{G}^2
=
[\lambda-1+(m_i+1)R_\mathbb{G}]^2 + 4\lambda m_i R_\mathbb{G}^2 - 4m_iR_\mathbb{G}^2 - 4(\lambda-1)m_iR_\mathbb{G}.
$$

\noindent Since $\lambda>1$ and $R_\mathbb{G}>1$,
$$
4\lambda m_iR_\mathbb{G}^2-4m_iR_\mathbb{G}^2-4(\lambda-1)m_iR_\mathbb{G}
=
4R_\mathbb{G}(\lambda-1) m_i(R_\mathbb{G}-1) >0.
$$

\noindent So $[\lambda-1-(m_i-1)R_\mathbb{G}]^2 + 4\lambda m_i R_\mathbb{G}^2 > [\lambda-1+(m_i+1)R_\mathbb{G}]^2$. By \eqref{pf:thm3.14},
$$
1
>
\sum_{i=1}^r \frac{(1-\lambda)+ (m_i-1) R_\mathbb{G} +[\lambda-1+(m_i+1)R_\mathbb{G}]}{2m}
=
\sum_{i=1}^r \frac{m_i \, R_\mathbb{G}}{m}
=
R_\mathbb{G},
$$

\noindent contradicting the assumption $R_\mathbb{G}>1$. Hence $U(o,\, o\, |\, R_\mathbb{G})<1$ in this case as well.

{\bf Step 3.} Let $\phi_i(z) := m-1+\lambda-(m_i-1)z$ for $1\le i\le r$, and
\begin{equation}
  \label{e:FzU}
F(z,\, U)
:=
\frac{1}{2m}\sum_{i=1}^r \left\{-(\phi_i(z)-mU)+[(\phi_i(z)-mU)^2+4\lambda m_iz^2]^{1/2}\right\}.
\end{equation}

\noindent Then $U(z):= U(o,\, o\, | \, z)$ solves the equation $U=F(z,\, U)$, $|z|<R_U$, and $\rho(\lambda)^{-1}$ is the smallest positive number $z$ such that $\frac{\partial F}{\partial U} (z,\, U(z))=1$. Therefore, to obtain $\rho(\lambda)<1,$ it suffices to prove that
$$
\left| \frac{\partial F}{\partial U} (1,\, U(1))\right|
<
1.
$$

To prove this, we observe that
\begin{align*}
    F(1,\, 0)
 &= \frac{1}{2m} \sum_{i=1}^r \left\{-(m-m_i+\lambda) + [(m-m_i+\lambda)^2+4\lambda m_i]^{1/2} \right\}>0,
    \\
    F(1,\, 1)
 &= \frac{1}{2m} \sum_{i=1}^r \left\{m_i-\lambda + [(m_i-\lambda)^2+4\lambda m_i]^{1/2} \right\}
    =
    \frac{1}{2m} \sum_{i=1}^r \{ m_i-\lambda+m_i+\lambda\}
    =1.
\end{align*}

\noindent Moreover,
\begin{align*}
    \frac{\partial F}{\partial U} (1,\, U)
 &= \frac{r}{2} - \frac{1}{2} \sum_{i=1}^r\frac{m-m_i+\lambda-mU}{[(m-m_i+\lambda-mU)^2+4\lambda m_i]^{1/2}}
    > 0,
    \\
    \frac{\partial^2F}{\partial U^2} (1,\, U)
 &= \frac{m}{2} \sum_{i=1}^r \frac{4\lambda m_i}{\{(m-m_i+\lambda-mU)^2+4\lambda m_i\}^{3/2}}
    > 0 .
\end{align*}

\noindent Hence $F(1,\, U)$ is strictly increasing and convex in $U\in\mathbb{R}$. By (\ref{(3.5)}), for any $\lambda \in (0, \, \lambda_c(G))$,
$$
\frac{\partial F}{\partial U} (1,\, 1)
=
\frac{r}{2} - \frac{1}{2} \sum_{i=1}^r \frac{-m_i+\lambda}{[(m_i-\lambda)^2+4\lambda m_i]^{1/2}}
=
\sum_{i=1}^r\frac{m_i}{m_i+\lambda}
>1.
$$

\noindent As a consequence, $U(1)$ is the smallest positive solution to $U=F(1,U)$ and $0<\frac{\partial F}{\partial U}(1,\, U(1))<1$. 
Therefore we have proved that $\rho(\lambda) < 1$. 

{\bf Step 4.} Now we  prove (\ref{(3.6)}).

If $m_i=1$ for all $i$, then $G$ is the $r$-regular tree, and by Theorem \ref{thm2.1}(ii), $\rho(0+)=\rho(0)=0$, so (\ref{(3.6)}) holds.

\noindent Assume that 
$$
\max_{1\le i\le r}m_i
=
m_{i_*}>1
\ \mbox{for some}\
i_* \in \{1, \ldots, r\}.
$$

\noindent For any $1\le i\le r$ and $n \geq 3$,
let
$$
A_i(n)
:=
\{X_0=o, \ \langle X_1\rangle=\langle X_2\rangle=\cdots =\langle X_{n-1}\rangle=i, \ X_n=o \}.
$$
Then
$$
p_\lambda^{(n)}(o,\, o)
\ge
\mathbb{P}_o[A_{i_*}(n)]
=
\frac{m_{i_*}}{m} \left(\frac{m_{i_*}-1}{m-1+\lambda}\right)^{n-2} \frac{\lambda}{m-1+\lambda} ,
$$

\noindent which implies that
$$
\rho(\lambda)
\ge
\lim_{n\to\infty} \{\mathbb{P}_o[A_{i_*}(n)]\}^{1/n}
=
\frac{m_{i_*}-1}{m-1+\lambda}.
$$

\noindent Consequently,
$$
\liminf_{\lambda\downarrow 0} \rho(\lambda)
\ge
\frac{m_{i_*}-1}{m-1}.
$$

It remains to prove that $\limsup_{\lambda\downarrow 0} \rho(\lambda)
\le \frac{m_{i_*}-1}{m-1}$. Let us make a few simple observations concerning the transition probability of $\{ |X_n|\}_{n\ge 0}$. Let $\ell\in\mathbb{N}$ and let $k \in\mathbb{N}$.

For any $j\in \{ 1, \ldots, r\}$ such that $\mathbb{P}_o [ \langle X_k\rangle=j, \ |X_k|=\ell]>0$, we have
$$
\mathbb{P}_o[ |X_{k+1}|=\ell\ \Big| \ \langle X_k\rangle=j, \ |X_k| =\ell, \ |X_{k-1}|, \ldots, |X_0|]
=
\frac{m_j-1}{m-1+\lambda}
\le
\frac{m_{i_*}-1}{m-1+\lambda},
$$

\noindent so that
$$
\mathbb{P}_o[ |X_{k+1}|=\ell\ \Big| \ |X_k| =\ell, \ |X_{k-1}|, \ldots, |X_0|]
\le
\frac{m_{i_*}-1}{m-1+\lambda} .
$$

\noindent On the other hand,
$$
\mathbb{P}_o[ |X_{k+1}|=\ell-1\ \Big| \ |X_k| =\ell, \ |X_{k-1}|, \ldots, |X_0|]
=
\frac{\lambda}{m-1+\lambda},
$$

\noindent and trivially,
$$
\mathbb{P}_o[ |X_{k+1}|=\ell+1\ \Big| \ |X_k| =\ell, \ |X_{k-1}|, \ldots, |X_0|]
\le
1.
$$

\noindent For any $ n \ge 3$, let $\mathcal{S}_n$ denote the set of all vectors $\vec{s} := \{s_k\}_{1\le k\le n}$ such that
$$
s_1=1,\ s_n=-1,\
s_k \in \{-1,\, 0,\, +1\},\ \sum_{j=1}^k s_j\ge 0, 1\le k\le n-1, \ \sum_{j=1}^n s_j=0.
$$

\noindent For $\vec{s}\in\mathcal{S}_n$, let
$$
a_+(\vec{s})
:=
\# \{k\le n:\ s_k=+1\},
\
a_- (\vec{s})
:=
\# \{k\le n:\ s_k=-1\},
\
a_0(\vec{s})
:=
\# \{k\le n:\ s_k=0\}.
$$

\noindent Clearly, $a_+ (\vec{s}) = a_- (\vec{s})$, $2a_- (\vec{s})+ a_0(\vec{s})=n$. Moreover, if $|X_n| =0$, then $\{ |X_{k}| - |X_{k-1}| \}_{1 \le k\le n}\in\mathcal{S}_n$.

By our discussions on transition probabilities of $\{ |X_n|\}_{n\ge 0}$, it is seen that for $3\le n$ and $\vec{s}\in\mathcal{S}_n$,
$$
\mathbb{P}_o [ \, |X_n| =0\, \Big| \, \{|X_k| - |X_{k-1}|\}_{1\le k\le n}=\vec{s}\in\mathcal{S}_n \, ]
\le
1^{a_+(\vec{s})}
\left( \frac{\lambda}{m-1+\lambda} \right)^{a_-(\vec{s})} \left( \frac{m_{i_*}-1}{m-1+\lambda} \right)^{a_0(\vec{s})} .
$$

\noindent For sufficiently small $\lambda>0$, we have $\frac{\lambda}{m-1+\lambda} \le (\frac{m_{i_*}-1}{m-1+\lambda})^2$, so that
$$
\left( \frac{\lambda}{m-1+\lambda} \right)^{a_-(\vec{s})} \left( \frac{m_{i_*}-1}{m-1+\lambda} \right)^{a_0(\vec{s})}
\le
\left( \frac{m_{i_*}-1}{m-1+\lambda} \right)^{2a_-(\vec{s})} \left( \frac{m_{i_*}-1}{m-1+\lambda} \right)^{a_0(\vec{s})}
=
\left( \frac{m_{i_*}-1}{m-1+\lambda} \right)^n.
$$

\noindent Consequently,
$$
\mathbb{P}_o [ \, |X_n| =0\, ]
\le
\left( \frac{m_{i_*}-1}{m-1+\lambda} \right)^n.
$$

\noindent Hence,
$$
\limsup_{\lambda\downarrow 0} \rho(\lambda)
\le
\limsup_{\lambda\downarrow 0}\frac{m_{i_*}-1}{m-1+\lambda}
=
\frac{m_{i_*}-1}{m-1}
=
\frac{\max_{1\le i\le r} m_i-1}{m-1},
$$

\noindent completing the proof of \eqref{(3.6)}.
\qed

\subsection{Proof of Theorem \ref{thm3.1}}
\label{pfthm3.1}

Recall that $G$ is the free product of two complete graphs $K_{m_1 + 1}$ and $K_{m_2+1}$ and that $(X_n)_{n=0}^{\infty}$ is the $\lambda$-biased random walk on $G$. Recall that 
$\lambda_c(G) = \sqrt{m_1 m_2}$. 

Define
$$
f(x)=
\begin{cases}
  1 &{\rm if}\ x=o,
      \\
      \frac{m_2-\lambda}{m-1+\lambda} &{\rm if}\ x\neq o,\ \langle x\rangle =1,
      \\
       \frac{m_1-\lambda}{m-1+\lambda} &{\rm if}\ x\neq o,\ \langle x\rangle =2.
\end{cases}
$$

\noindent Then $\{|X_n|-|X_{n-1}|-f(X_{n-1})\}_{n=1}^\infty$ is a martingale-difference sequence. It follows from the strong law of large numbers for uncorrelated random variables (\cite[Theorem 13.1]{LR-PY2016}) that
$$
\lim_{n\to\infty} \frac{1}{n}\left(|X_n|-\sum_{k=0}^{n-1} f(X_k) \right)
=
0
\qquad
\hbox{\rm a.s.} 
$$

\noindent Note that $\sum_{k=0}^{n-1} f(X_k) = \sum_{k=0}^{n-1} f(o) \, I_{\{X_k=o\}}+\sum_{i=1}^2\sum_{k=0}^{n-1} f(i) \, I_{\{\langle X_k\rangle=i\}}$. Since the walk is transient $\frac{1}{n}\sum_{k=0}^{n-1}I_{\{X_k=o\}}\to 0$ a.s.\ for $0\le \lambda < \lambda_c(G)$. 
Consequently, we have 
\begin{equation}
    \label{(3.2)}
    \lim_{n\to\infty} \frac{1}{n}
    \left( |X_n|
           - \frac{m_2-\lambda}{m-1+\lambda} \sum_{k=0}^{n-1}I_{\{\langle X_k\rangle=1\}}
           - \frac{m_1-\lambda}{m-1+\lambda} \sum_{k=0}^{n-1}I_{\{\langle X_k\rangle=2\}}
    \right)
    =0 \quad \text{a.s.}
  \end{equation}
  



For any $\lambda\in [0,\, \infty)$, let
$$
F(\lambda)
:=
\frac{m_2-\lambda}{\lambda+m-1}\, \frac{m_1+\lambda}{2\lambda+m}
+
\frac{m_1-\lambda}{\lambda+m-1}\, \frac{m_2+\lambda}{2\lambda+m}
=
\frac{2m_1m_2-2\lambda^2}{(\lambda+m-1)(2\lambda+m)}.
$$
Note that $\lambda\mapsto F(\lambda)$ is strictly decreasing on $[0, \, \infty)$.

\begin{lem}
 \label{lem3.5}

 For any $0\le\lambda < \lambda_c(G)$, the speed $\mathcal{S}(\lambda) :=\lim_{n\to\infty}\frac{|X_n| }{n}$ exists almost surely, is deterministic and equals $F(\lambda)$. In particular,
$$
\mathcal{S}(\lambda)>0
\ \mbox{is smooth and strictly decreasing in $\lambda\in [0, \, \lambda_c(G))$, and}\
\lim_{\lambda\uparrow\lambda_c(G)}\mathcal{S}(\lambda)
=
0.
$$

\end{lem}

\noindent {\it Proof.} {\bf Step 1.} Consider the process $(|X_n|, \langle X_n\rangle)_{n=0}^\infty$. For any type $1$ (resp.\ type $2$) vertex $x$, all its $m_2$ (resp.\ $m_1$) neighbours in $\partial B_G(|x|+1)$ are of type $2$ (resp.\ $1$), and its unique neighbour $x_{-}$ in $\partial B_G(|x| -1)$ is of type $2$ (resp.\ type $1$) if $|x| \ge 2$, and is $o$ if $|x| =1$. The vertex $o$ has exactly $m_1$ type $1$ neighbours and $m_2$ type $2$ neighbours in $\partial B_G(1)$.
The process $(|X_n|, \langle X_n\rangle)_{n=0}^\infty$ is a Markov chain on state space $(\mathbb{N}\times \{1, \, 2\})\cup\{(0,\, 0)\}$
with transition probability function $q(\, \cdot\, , \, \cdot\, )$ given by
\begin{eqnarray*}
&&q((0,0),\, (1,1))=\frac{m_1}{m},~q((0,0),\, (1,2))=\frac{m_2}{m},\\
&&q((1,1),\, (0,0))=\frac{\lambda}{m-1+\lambda},\ q((1,2),\, (0,0))=\frac{\lambda}{m-1+\lambda},\\
&&q((1,1),\, (1,1))=\frac{m_1-1}{m-1+\lambda},~q((1,1),\, (2,2))=\frac{m_2}{m-1+\lambda}, \\
&&q((1,2),\, (1,2))=\frac{m_2-1}{m-1+\lambda},~q((1,2),\, (2,1))=\frac{m_1}{m-1+\lambda};
\end{eqnarray*}
and for any $k\ge 2$,
\begin{eqnarray*}
&&q((k,1), \, (k,1))=\frac{m_1-1}{m-1+\lambda},\ q((k,2), \, (k,2))=\frac{m_2-1}{m-1+\lambda},\\
&&q((k,1), \, (k-1,2))=\frac{\lambda}{m-1+\lambda},\ q((k,1), \, (k+1,2))=\frac{m_2}{m-1+\lambda},\\
&&q((k,2), \, (k-1,1))=\frac{\lambda}{m-1+\lambda},\ q((k,2), \, (k+1,1))=\frac{m_1}{m-1+\lambda}.
\end{eqnarray*}

{\bf Step 2.}
Define, for $i\in\{1, \, 2\}$,
$$
\sigma_1^i := \inf \{n\ge 0: \ \langle X_n\rangle=i \},
\qquad
\tau_1^i := \inf \{n>\sigma_1^i: \ \langle X_n\rangle\neq i \},
$$

\noindent and recursively for any $k\in\mathbb{N}$,
$$
\sigma_{k+1}^i := \inf \{n>\tau_{k}^i:\ \langle X_n\rangle=i \},
\qquad
\tau_{k+1}^i := \inf \{n>\sigma_{k+1}^i:\ \langle X_n\rangle\neq i \}.
$$

\noindent Set
$$
p_1 := \frac{m_1-1}{m-1+\lambda},
\qquad
p_2 := \frac{m_2-1}{m-1+\lambda}.
$$

By Step 1 and the strong Markov property, all stopping times $\sigma_k^i$ and $\tau_k^i$ are finite, and $\{\tau_k^i-\sigma_k^i-1\}_{k\ge 1}$ is an i.i.d.\ sequence with $\mathbb{P}( \tau_k^i-\sigma_k^i-1 =j) = p_i^j (1-p_i)$ for $j\ge 0$. In particular, $\mathbb{E}(\tau_k^i-\sigma_k^i-1)=\frac{p_i}{1-p_i}$.

Notice that 
for any $n\ge 1+\sigma_1^i$, there exists a unique random integer $k_n^i$ such that $\sigma_{k_n^i}\le n-1 <\sigma_{k^i_{n+1}}$. Therefore, for any $n\ge 1+\sigma_1^1\vee \sigma_1^2$ and $i\in \{ 1, \, 2\}$,
$$
\frac{1}{n}\sum_{j=1}^{k_n^i-1} (\tau_j^i-\sigma_j^i)
\le
\frac{1}{n}\sum_{k=0}^{n-1}I_{ \{\langle X_k\rangle=i \}}
\le
\frac{1}{n}\sum_{j=1}^{k_n^i} (\tau_j^i-\sigma_j^i).
$$

Since $\{\tau_k^i-\sigma_k^i-1\}_{k\ge 1}$ is i.i.d.\ with $\mathbb{E}(\tau_1^i-\sigma_1^i)<\infty$, we have $\frac{1}{n} (\tau^i_{k_n^i} -\sigma^i_{k_n^i}) \to 0$ a.s. Consequently, for $i\in \{ 1, \, 2\}$,
$$
\frac{1}{n}\sum_{j=1}^{k_n^i} (\tau_j^i-\sigma_j^i)
-
\frac{1}{n}\sum_{k=0}^{n-1} I_{\{\langle X_k\rangle=i\}}
\to 0
\qquad
\hbox{\rm a.s.}
$$

{\bf Step 3.} Almost surely,
\[\lim_{n\to\infty}\frac{k_n^1}{n}=\lim_{n\to\infty}\frac{k_n^2}{n}=
     \frac{(m_1+\lambda)(m_2+\lambda)}{(m+2\lambda)(m-1+\lambda)}.\]

\noindent Indeed, $\frac{1}{n}\sum_{k=0}^{n-1}I_{\{ X_k=o \}} \to 0$ a.s., thus
$$
\lim_{n\to\infty} \left\{\frac{1}{n}\sum_{k=0}^{n-1}I_{\{\langle X_k\rangle= 1\}}+\frac{1}{n}\sum_{k=0}^{n-1}I_{\{\langle X_k\rangle=2 \}}\right\}=1
\qquad
\hbox{\rm a.s.}
$$

\noindent By Step 2, this implies that
\begin{equation}
    \lim_{n\to\infty} \frac{1}{n} \left( \sum_{j=1}^{k_n^1} (\tau_j^1-\sigma_j^1 ) + \sum_{j=1}^{k_n^2} (\tau_j^2-\sigma_j^2 )\right)
    =
    1
    \qquad
    \hbox{\rm a.s.}
    \label{pf:lem3.12}
\end{equation}
On the other hand, each $\{\tau_k^i-\sigma_k^i-1\}_{k\ge 1}$ is an i.i.d.\ sequence with $\mathbb{E}( \tau_k^i-\sigma_k^i-1 ) = \frac{p_i}{1-p_i}$, thus by the strong law of large numbers, for $i\in\{1, \, 2\}$,
$$
\lim_{n\to\infty} \frac{1}{k_n^i} \sum_{j=1}^{k_n^i} (\tau^i_j-\sigma^i_j )
=
1+\frac{p_i}{1-p_i}
=
\frac{1}{1-p_i}
\qquad
\hbox{\rm a.s.}
$$

\noindent In view of \eqref{pf:lem3.12}, we obtain:
\begin{equation}
    \lim_{n\to\infty} \left\{\frac{k_n^1}{n} \, \frac{1}{1-p_1}+\frac{k_n^2}{n}\, \frac{1}{1-p_2} \right\}
    =
    1
    \qquad
    \hbox{\rm a.s.}
    \label{(3.3)}
\end{equation}

Observe that $\langle X_{\tau_k^i} \rangle$ is either $j$ for $j\in\{1,\, 2\}\backslash\{i\}$ or $X_{\tau_k^i}=o$, and that when $X_{\tau_k^i}=o$, $X_{\tau_k^i+1}$ must be of type  $1$ or $2$. Since
$$
\frac1n \sum_{k=0}^{n-1} I_{\{X_k=o\}}
\to
0
\qquad
\hbox{\rm a.s.}
$$

\noindent it implies almost surely, that for $n\to \infty$, $k_n^1$ (the number of jumps of $\RW_{\lambda}$ from $o$ or type $2$ vertex to type $1$ vertex up to time $n-1$) differs by $o(n)$ from $k_n^2$ (the number of jumps of $\RW_{\lambda}$ from $o$ or type $1$ vertex to type $2$ vertex up to time $n-1$). In other words, 
$\frac{k_n^1-k_n^2}{n} \to 0$ a.s. In view of \eqref{(3.3)}, we get
$$
\lim_{n\to\infty}\frac{k_n^1}{n}
=
\lim_{n\to\infty}\frac{k_n^2}{n}
=
\frac{(m_1+\lambda)(m_2+\lambda)}{(m+2\lambda)(m-1+\lambda)}
\qquad
\hbox{\rm a.s.}
$$

{\bf Step 4.} By Steps 2 and 3, for $i\in\{1, \, 2\}$,
$$
\frac{1}{n}\sum_{k=0}^{n-1} I_{\{\langle X_k\rangle=i\}}
\to
\frac{m_i+\lambda}{m+2\lambda}
\qquad
\hbox{\rm a.s.}
$$

\noindent This and \eqref{(3.2)} complete the proof of this lemma.
\qed 


\bigskip

The next lemma concerns the non-Liouville property of $\RW_{\lambda}$ with $0\le\lambda<\lambda_c(G)$.

\begin{lem}
 \label{lem3.6}

 For any $0\le\lambda<\lambda_c(G)$, $\RW_{\lambda}$ has a non-constant bounded harmonic function.

\end{lem}

\pf Take $y\in \partial B_G(1)$ with $\langle y\rangle=2$. Let $A$ be the induced subgraph consisting of $y$ and 
all words (vertices) of forms $yw$. Let $(X_n)_{n=0}^{\infty}$ be $\RW_{\lambda}$ on $G$, and let $\mathbb{P}_z$ denote the law of $(X_n)_{n=0}^\infty$ starting at $z$. Notice that every vertex $z\in G$ is a cutpoint in the sense that $G\backslash \{z\}$ has two disjoint connected components. By the transience of 
$\RW_{\lambda}$, $\lim_{n\to \infty}I_{\{X_n\in A\}}$ exists $\mathbb{P}_z$-a.s.

For any vertex $z$ of $G$, let
$$
f(z)
:=
\mathbb{P}_z [(X_n)_{n=0}^{\infty}~ \mbox{ends up in $A$}].
$$

\noindent  Then $f$ is a bounded harmonic function. Let $x\in \partial B_G(1)$ with $\langle y\rangle =1$. Let
$$
a
:=
\mathbb{P}_x[ (X_n)_{n=0}^{\infty}~ \mbox{never hits $o$}].
$$

\noindent Since the walk is transient, we have $a\in (0,\, 1)$, and $f(x)=(1-a)f(o)$. Note that $(G,\, o)$ is quasi-spherically symmetric, so the transience of the walk implies $f(o)>0$. Hence $f$ is a non-constant harmonic function.
\qed

\bigskip

\noindent{\bf Proof of Theorem \ref{thm3.1}.} By
Lemmas \ref{lem3.5}-\ref{lem3.6}, we obtain Theorem \ref{thm3.1}(i)-(ii). It remains to prove Theorem \ref{thm3.1}(iii).


{\bf Step 1. Computation of $\rho(\lambda)$.}
Recall from Step 3 in the proof of Theorem~\ref{thm3.2} (Section \ref{subs:free_product_complete_graphs}) that $U(z) := U(o, \,
o\, | \, z)$ solves the equation $U = F(z,\, U)$, $|z| < R_{\mathbb{G}}$, and
$z_0 := \rho(\lambda)^{-1}$ is the smallest positive number $z$ such that
$\frac{\partial F}{\partial U}(z,\, U(z)) = 1$, where the function $F(z,\, U)$
is defined by \eqref{e:FzU} with $r = 2$:
$$
F(z,\, U)
:=
\frac{1}{2m} \sum_{i=1}^2 \{ -(\phi_i(z)-mU) + [(\phi_i(z)-mU)^2 + 4\lambda m_i z^2]^{1/2} \}.
$$
\noindent Since
\[
  \frac{\partial F}{\partial U}(z,\, U)
  =
  1 + \frac{1}{2} \sum_{i=1}^2 \frac{- \left( \phi_i(z) - m U \right)}{[ (
        \phi_i(z) - m U)^2 + 4 \lambda m_i z^2]^{1/2}},
\]
we have
\begin{equation}
  \label{e:phiz0}
  \frac{\phi_1(z_0) - m U(z_0)}{[(
        \phi_1(z_0) - m U(z_0))^2 + 4 \lambda m_1 z_0^2]^{1/2}}
  =
  \frac{m U(z_0) - \phi_2(z_0)}{[(
        \phi_2(z_0) - m U(z_0))^2 + 4 \lambda m_2 z_0^2]^{1/2}},
\end{equation}

\noindent which implies
$$
  \frac{\phi_1(z_0) - m U(z_0)}{\sqrt{4 \lambda m_1 z_0^2}}
  =
  \frac{m U(z_0) - \phi_2(z_0)}{\sqrt{4 \lambda m_2 z_0^2}}.
$$
Recall that $\phi_i(z) = m + \lambda - 1 - (m_i - 1) z$. This yields
\begin{equation}
  \label{e:Uz0}
  m U(z_0)
  =
  m + \lambda - 1 - ( \sqrt{m_1 m_2} - 1 ) z_0;
\end{equation}

\noindent hence $\phi_i (z_0) - m U(z_0) = (\sqrt{m_1 m_2} - m_i)z_0$. Consequently,
\begin{eqnarray}
    F(z_0,\, U(z_0))
 &=& \frac{1}{2m} \sum_{i=1}^2 \{ -(\sqrt{m_1 m_2} - m_i)z_0 + [(\sqrt{m_1 m_2} - m_i)^2z_0^2 + 4\lambda m_i z_0^2]^{1/2} \}
    \nonumber
    \\
 &=& - \Big( \frac{\sqrt{m_1 m_2}}{m} - \frac12 \Big) z_0 + \frac{1}{2m} \sum_{i=1}^2 m_i^{1/2} [(\sqrt{m_1} - \sqrt{m_2})^2 + 4\lambda]^{1/2} z_0
    \nonumber
    \\
 &=& - \Big( \frac{\sqrt{m_1 m_2}}{m} - \frac12 \Big) z_0 + \frac{1}{2m} [(m_1-m_2)^2 + 4\lambda(\sqrt{m_1} + \sqrt{m_2})^2]^{1/2} z_0 \, .
    \label{e:F(z0,U(z0))}
\end{eqnarray}

\noindent On the other hand, $F(z_0,\, U(z_0)) = U(z_0)$, which is $\frac{m + \lambda - 1}{m} - \frac{\sqrt{m_1 m_2} - 1}{m} z_0$ (by \eqref{e:Uz0}). Combining this with \eqref{e:F(z0,U(z0))} yields
\begin{equation}
  \label{e:z0}
  \rho(\lambda)^{-1}
  =
  z_0
  =
  \frac{2(m+\lambda-1)}{m-2 + [ (m_1 - m_2)^2 + 4 \lambda ( \sqrt{m_1} + \sqrt{m_2} )^2 ]^{1/2}}.
\end{equation}
Taking limit $\lambda \to 0+$, we have
\[
  \lim_{\lambda \to 0+} \rho(\lambda)
  =
  \frac{(m_1 \vee m_2) - 1}{m-1}.
\]

{\bf Step 2. Strictly increasing property for $\rho(\lambda)$.}
By a change of variables
\[
  x
  =
  m-2 + [ (m_1 - m_2)^2 + 4 \lambda ( \sqrt{m_1} + \sqrt{m_2} )^2 ]^{1/2},
\]
($\lambda = \frac{(x-m+2)^2 - (m_1 - m_2)^2}{4( \sqrt{m_1} + \sqrt{m_2} )^2}$), we see that
\[
  z_0
  =
  \frac{1}{2(\sqrt{m_1} + \sqrt{m_2})^2} \, \left[ x + \frac{4 (m-1 + \sqrt{m_1
        m_2})^2}{x} - 2(m-2) \right],
\]
which is strictly decreasing in $x < 2 (m-1 + \sqrt{m_1m_2})$, i.e., $\lambda
< \sqrt{m_1 m_2} = \lambda_c(G)$. Thus $\rho(\lambda)$ is strictly increasing in
$\lambda \in (0,\, \sqrt{m_1m_2})$.

{\bf Step 3. Asymptotics for $p^{(n)}_\lambda (o,\, o)$.}
Write for simplicity $\mathbb{G}(z) := \mathbb{G}(o,\, o\, |\, z)$. Note that $U(z) = \frac{\G(z)}{\G(z) - 1}$, $|z| < R_{\mathbb{G}}$. We have from $U(z) = F(z,\, U(z))$ that
\begin{equation}
  \label{e:Hz}
  2(m+\lambda-1) - (m-2)z
  =
  \sum_{i=1}^2 \left[ \left(
        \phi_i(z) - \frac{m \G(z)}{\G(z) - 1} \right)^2 + 4 \lambda m_i z^2 \right]^{1/2}.
\end{equation}
Set
\[
  \Psi(u,\, v)
  :=
  2(m+1-\lambda) - (m-2) u - \sum_{i=1}^2 \left[ \left( \phi_i(u) - \frac{m v}{v-1} \right)^2 + 4\lambda m_i u^2 \right]^{1/2}.
\]
Notice that $\Psi(z, \, \G(z)) = 0$. By \eqref{e:phiz0}, there exists $\theta_0 \in (0, \,
\pi)$ such that
\[ \cos \theta_0
  =
  \frac{\phi_1(z_0) - m U(z_0)}{[(\phi_1(z_0) - m U(z_0))^2 + 4 \lambda m_1 z_0^2]^{1/2}}
  =
  \frac{m U(z_0) - \phi_2(z_0)}{[(\phi_2(z_0) - m U(z_0))^2 + 4 \lambda m_2 z_0^2]^{1/2}}.
\]
By direct computations, we have
\begin{align*}
  \frac{\partial \Psi}{\partial v}(z_0,\, \G(z_0))
  & =
      0, \\
  \frac{\partial^2 \Psi}{\partial v^2}(z_0, \, \G(z_0))
  & =
  - m^2 \sum_{i=1}^2 \frac{(\G(z_0) - 1)^{-4} \sin^2 \theta_0 }{[(
    \phi_i(z_0) - m U(z_0))^2 + 4 \lambda m_i z_0^2]^{1/2}}
  \neq 0 \\
  \frac{\partial \Psi}{\partial u}(z_0,\, \G(z_0))
  & = -(m-2) + (m_1 - m_2) \cos \theta_0 - 2\sqrt{\lambda}(\sqrt{m_1} + \sqrt{m_2}) \sin \theta_0
    \neq 0.
\end{align*}
Applying the method of Darboux (see \cite{BE1974} Theorem 5) as in the proof of Lemma \ref{lem2.6}, we obtain the desired asymptotics for $p^{(n)}_\lambda (o,\, o)$.
\qed

\appendix

\section{Proof of Lemma~\ref{lem2.6} 
}
\label{s:pfthm2.2}


\begin{proof}[Proof of Lemma~\ref{lem2.6}] The lemma holds trivially for $\lambda=0.$ So assume $\lambda>0.$ Notice that $\RW_\lambda$ $(X_n)_{n=0}^{\infty}$ must return to $o$ in even steps, and that $\{|X_n| \}_{n=0}^{\infty}$ with $|X_0| =0$
is a Markov chain on $\mathbb{Z}_{+}$ with transition probabilities given by
$$
p(x, \, y)
=
\left\{\begin{array}{cl}
1 &{\rm if}\ x=0, \ y=1\\
      \frac{\lambda}{d-1+\lambda} &{\rm if}\ y=x-1\ \mbox{and}\ x\neq 0,\\
     \frac{d-1}{d-1+\lambda} &{\rm otherwise}.
   \end{array}
\right.
$$

\noindent Recall for any $n\in \mathbb{N}$ and $k\in\mathbb{Z}_{+}$,
$$
f^{(2n)}_\lambda (o,\, o)
=
\mathbb{P}_o\left(\tau_o^{+}=2n\right),
\qquad
f^{(2n-1)}_\lambda(o,\, o)=0,\ \lambda\in (0,\, \infty),
$$

\noindent and the $k$th Catalan number given by $c_k = \frac{1}{k+1}{{2k}\choose{k}}$, with the associated related generating function
\begin{equation}
    \mathcal{C}(x)
    :=
    \sum_{k=0}^{\infty}c_kx^k
    =
    \frac{1-\sqrt{1-4x}}{2x},
    \qquad
    x\in \left[-\frac{1}{4}, \, \frac{1}{4}\right].
    \label{Catalan_generating_fct}
\end{equation}

\noindent Note the number of all $2n$-length nearest-neighbour paths $\gamma=w_0w_1\cdots w_{2n}$ on $\mathbb{Z}_{+}$ such that
$$
w_0=w_{2n}=0,
\qquad
w_j\ge 1,\ 1\le j\le 2n-1
$$

\noindent is precisely $c_{n-1}$. Hence for any $\lambda>0,$
$$
f_\lambda^{(2n)}(o, \, o)
=
c_{n-1}\left(\frac{d-1}{d-1+\lambda}\right)^{n-1}\left(\frac{\lambda}{d-1+\lambda}\right)^n,\ n\in\mathbb{N},
$$

\noindent which readily yields \eqref{f_regular_tree} by means of Stirling's formula.

By definition, for $\lambda>0$,
\begin{align*}
    U_\lambda(o,\, o\, |\, z)
 &=\sum_{n=1}^\infty f^{(2n)}_\lambda(o,\, o)z^{2n}
    =
    \sum_{n=1}^\infty c_{n-1}
    \left(\frac{d-1}{d-1+\lambda}\right)^{n-1}
    \left(\frac{\lambda}{d-1+\lambda}\right)^n z^{2n}
    \\
 &=\frac{\lambda}{d-1+\lambda} z^2 \, \mathcal{C}\left(\frac{\lambda(d-1)z^2}{(d-1+\lambda)^2}\right) ,
\end{align*}

\noindent which, in view of \eqref{Catalan_generating_fct}, implies that for $|z| \le \frac{d-1+\lambda}{2\sqrt{\lambda(d-1)}}$,
\begin{equation}
    U_\lambda(o,\, o\, |\, z)
    =
    \frac{(d-1+\lambda)-\sqrt{(d-1+\lambda)^2-4\lambda(d-1)z^2}}{2(d-1)} .
    \label{U_regular_tree}
\end{equation}

\noindent Taking $z=1$ gives that
$$
\theta_{\mathbb{T}_d}(\lambda)
=
U_\lambda(o, \, o \, | \, 1)
=
\frac{\lambda\wedge (d-1)}{d-1}.
$$
Notice from \eqref{U_regular_tree} that when $0<\lambda\le d-1$,
$$
U_{\lambda}\left(o,\, o \, \Big| \, \frac{d-1+\lambda}{2\sqrt{\lambda(d-1)}} \right)
=
\frac{d-1+\lambda}{2(d-1)}
\le 1.
$$

\noindent Hence, for $|z|<\frac{d-1+\lambda}{2\sqrt{\lambda(d-1)}}$ and $0<\lambda\le d-1$,
\begin{align}
    \mathbb{G}_\lambda(o, \, o\, |\, z)
 &=\frac{1}{1-U_\lambda(o, \, o\, |\, z)}
    \nonumber
    \\
 &=\frac{2(d-1)}{2(d-1)- (d-1+\lambda)+\sqrt{(d-1+\lambda)^2-4\lambda(d-1)z^2}} .
    \label{G_regular_tree}
\end{align}

\noindent This implies that the convergence radius for $\mathbb{G}_\lambda(o, \, o\, |\, z)$ is $\frac{d-1+\lambda}{2\sqrt{\lambda(d-1)}}$. In other words,
$$
\rho(\lambda)
:=
\rho_{\mathbb{T}_d}(\lambda)
=
\frac{2\sqrt{\lambda(d-1)}}{d-1+\lambda},
\qquad
0<\lambda\le d-1.
$$

It remains to show \eqref{p_regular_tree} for $\lambda\in (0, \, d-1)$. Write $a(\lambda)=\frac{2(d-1)}{d-1+\lambda}$ and $b(\lambda)=\frac{d-1-\lambda}{d-1+\lambda}.$ Then for any $|z|\le R_{\mathbb{G}}(\lambda)=\frac{1}{\rho(\lambda)},$
$$
\mathbb{G}_\lambda(o, \, o\, |\, z)
=
\frac{2(d-1)}{d-1+\lambda}\frac{1}{\frac{d-1-\lambda}{d-1+\lambda}+\sqrt{1-\rho(\lambda)^2z^2}}
=
\frac{a(\lambda)}{b(\lambda)+\sqrt{1-\rho(\lambda)^2z^2}}.
$$

\noindent Let
$$
\Phi(t)
:=
\Phi_\lambda(t)
=
\frac{-a(\lambda)b(\lambda)+\sqrt{a(\lambda)^2+\rho(\lambda)^2(1-b(\lambda)^2)t^2}}{1-b(\lambda)^2},
\qquad
t\in\mathbb{R}.
$$

\noindent Then for any $|z|\le R_{\mathbb{G}}(\lambda)$,
$$
\mathbb{G}_\lambda(o, \, o\, |\, z)
=
\Phi\left(z \, \mathbb{G}_\lambda(o, \, o\, |\, z)\right).
$$

\noindent Define
$$
\Psi(u,\, v)
:=
\Phi(uv)-v,
\qquad
u, \, v\in\mathbb{R}\, .
$$


\noindent Then
\begin{eqnarray*}
 && \frac{\partial \Psi(u,\, v)}{\partial v}
    \Big|_{(u,\, v)=\left( \frac{1}{\rho(\lambda)}, \, \mathbb{G}_{\lambda} (o, \, o \, | \, \frac{1}{\rho(\lambda)}) \right)}
    = 0,
    \\
    \,
    \\
 &&c_1(\lambda)
    :=
    \frac{\partial^2\Psi(u, \, v)}{\partial v^2}
    \Big|_{(u,\, v)=\left( \frac{1}{\rho(\lambda)}, \, \mathbb{G}_{\lambda} (o, \, o \, | \, \frac{1}{\rho(\lambda)}) \right) }
    =
    \frac{(d-1-\lambda)^3}{2(d-1)(d-1+\lambda)^2} \not= 0,
    \\
    \,
    \\
&&c_2(\lambda)
    :=
    \frac{\partial\Psi(u, \, v)}{\partial u}
    \Big|_{(u,\, v)=\left( \frac{1}{\rho(\lambda)}, \, \mathbb{G}_{\lambda} (o, \, o \, | \, \frac{1}{\rho(\lambda)}) \right) }
    =
    \frac{2\rho(\lambda)(d-1)}{d-1-\lambda}\not= 0.
\end{eqnarray*}

\noindent Applying the method of Darboux (see \cite{BE1974} Theorem 5), we obtain that
$$
p^{(2n)}_\lambda(o,\, o)
\sim
\left(\frac{c_1(\lambda)}{2\pi\rho(\lambda)c_2(\lambda)}\right)^{1/2}\rho(\lambda)^{2n}{(2n)}^{-3/2}
=
\frac{(d-1-\lambda)^2}{16(\pi\lambda)^{1/2}(d-1)^{3/2}}\rho(\lambda)^{2n}n^{-3/2}.
$$

\noindent The idea of using the method of Darboux to establish the asymptotics for $p^{(2n)}_\lambda(o,\, o)$ is not new. For example, in Woess~\cite{WW2000} Chapter III Section 17 pp.~181--189, examples of random walk on groups are given such that $p^{(n)}(o,\, o)\sim c\rho^{n}n^{-3/2}$ for some constant $c>0$. The exact value of $c$ is not known in general.

For $z\in (-1,\, 1),$ $\mathbb{G}_{d-1}(o, \, o\, |\, z) =\frac{1}{\sqrt{1-z^2}}
=\sum_{n=0}^{\infty} \frac{(2n)!}{2^{2n}(n!)^2}z^{2n}$. Thus 
$$
p_{d-1}^{(2n)}(o,\, o)
=
\frac{(2n)!}{2^{2n}(n!)^2}
\sim
\frac{1}{\sqrt{\pi n}}.
$$

\end{proof}

\flushleft{Zhan Shi\\
LPSM, Universit\'{e} Paris VI\\
4 place Jussieu, F-75252 Paris Cedex 05, France\\
E-mail: \texttt{zhan.shi@upmc.fr}}\\

\flushleft{Vladas Sidoravicius\\
NYU-ECNU Institute of Mathematical Sciences at NYU Shanghai\\
\& Courant Institute of Mathematical Sciences\\
New York, NY 10012, USA\\
E-mail: \texttt{vs1138@nyu.edu}}\\

\flushleft{He Song\\
  Department of Mathematical Science, Taizhou University\\
  Taizhou 225300, P. R. China\\
  Email: \texttt{tayunzhuiyue@126.com}}

\flushleft{Longmin Wang and Kainan Xiang\\
School of Mathematical Sciences, LPMC, Nankai University\\
Tianjin 300071, P. R. China\\
E-mails: \texttt{wanglm@nankai.edu.cn} (Wang)\\
\hskip 1.4cm \texttt{kainanxiang@nankai.edu.cn} (Xiang)}
\end{document}